\documentclass[11pt]{amsart}

\usepackage{latexsym}
\usepackage{amsmath}
\usepackage{amssymb}

\usepackage{epsfig}
\usepackage{epic}

\usepackage{enumerate}

\usepackage{color}

\newtheorem{theorem}{Theorem}[section]
\newtheorem{lemma}[theorem]{Lemma}
\newtheorem{corollary}[theorem]{Corollary}
\newtheorem{proposition}[theorem]{Proposition}
\newtheorem{sublemma}{}[theorem]

\newcommand{\cM}{{\mathcal M}}
\newcommand{\cN}{{\mathcal N}}

\newcommand{\blue}{\textcolor{black}}
\newcommand{\red}{\textcolor{black}}

\begin{document}

\title[Orchard Networks]{A Class of Phylogenetic Networks Reconstructable from Ancestral Profiles}

\author{P\'{e}ter L.\ Erd\H{o}s}
\address{Alfr\'{e}d R\'{e}nyi Institute of Mathematics, Hungarian Academy of Sciences, Budapest, Hungary}
\email{erdos.peter@renyi.mta.hu}

\author{Charles Semple}
\address{School of Mathematics and Statistics, University of Canterbury, Christchurch, New Zealand}
\email{charles.semple@canterbury.ac.nz}

\author{Mike Steel}
\address{School of Mathematics and Statistics, University of Canterbury, Christchurch, New Zealand}
\email{mike.steel@canterbury.ac.nz}

\thanks{The first author was supported in part by the National Research, Development and Innovation Office (NKFIH grants K~116769 and KH~126853). The second and third authors were supported by the New Zealand Marsden Fund (UOC1709).}



\date{\today}

\maketitle

\begin{abstract}
Rooted phylogenetic networks provide an explicit representation of the evolutionary history of a set $X$ of sampled species. In contrast to phylogenetic trees which show only speciation events, networks can also accommodate reticulate processes (for example, hybrid evolution, endosymbiosis, and lateral gene transfer). A major goal in systematic biology is to infer evolutionary relationships, and while phylogenetic trees can be uniquely determined from various simple combinatorial data on $X$, for networks the reconstruction question is much more subtle. Here we ask when can a  network be uniquely reconstructed from its `ancestral profile' (the number of paths from each ancestral vertex to each element in $X$). We show that reconstruction holds (even within the class of all networks) for a class of networks we call `orchard networks', and we provide a polynomial-time algorithm for reconstructing any orchard network from its ancestral profile. Our approach relies on establishing a structural theorem for orchard networks, which also provides for a fast (polynomial-time) algorithm to test if any given network is of orchard type. Since the class of orchard networks includes \blue{tree-sibling tree-consistent networks and} tree-child networks, our result \blue{generalise reconstruction results from 2008 and} 2009. Orchard networks allow for an unbounded number \blue{$k$} of reticulation vertices, in contrast to \blue{tree-sibling tree-consistent networks and} tree-child networks for which \blue{$k$} is at most \blue{$2|X|-4$ and} $|X|-1$, \blue{respectively}.
\end{abstract}

\section{Introduction}

Phylogenetic trees and networks have become a ubiquitous tool for representing  evolutionary relationships in systematics biology~\cite{fel04} and other areas of classification (for example, language evolution and epidemiology). From early sketches by Charles Darwin and Ernst Haeckel in the 19$^{{\rm th}}$ century, more complex and detailed trees are now revealing the finer details of portions of the `tree of life'. Today, biologists routinely build phylogenetic trees on hundreds of species, such as the recent tree of (nearly) all $\sim$10,000 species of birds~\cite{jet12}. Phylogenetic trees have a leaf set $X$ that consists of the sampled organisms (typically, a group of present-day species); the root of the tree represents the most recent common ancestor of the species in $X$. Current methods for inferring phylogenetic trees trees generally use genomic data from the species in $X$, and apply one of several possible reconstruction methods. While many of these methods are statistically based, they are ultimately founded on underlying combinatorial uniqueness results concerning trees~\cite{fel04, sem03}.

Although phylogenetic trees have proved a convenient representation for many groups of species including, for example, mammals and birds, in other domains of life evolution is not always described as a simple vertical process of speciation (where lineages split in two as new species form) and extinction. Instead, various reticulate processes allow for a `horizontal' component. Two main examples include the formation of hybrid species (such as in certain plant or fish species), and the exchange of genes between species in a process called lateral gene transfer (such as in bacteria). An additional reticulate process relevant to early life on earth is endosymbiosis in which organelles are incorporated into cells.

For these reasons, phylogenetic networks (acyclic directed graphs with a single root vertex and leaves forming the set $X$) have been proposed as a more flexible and accurate representation of evolutionary history~\cite{doo99, koo15}. Accordingly, there has been considerable recent interest in extending the mathematical foundation of phylogenetic tree reconstruction to networks~\cite{hus10}. This extension faces a number of mathematical obstacles. In particular, while trees can be encoded and reconstructed in several ways (for example, based on their associated system of clusters, path distances between pairs of leaves, and induced $3$-leaf subtrees), none of these approaches extends to networks, except for in very special cases~\cite{gam12, ier14, wil10}. This has led to various approaches being proposed, which usually involve one or more of the following:
\begin{enumerate}[(i)]
\item[(i)] not distinguishing between phylogenetic networks that are similar in a certain way~\cite{par15};

\item[(ii)] considering reconstruction only within a limited subclass of phylogenetic networks~\cite{bor18}; and

\item[(iii)] allowing types of information for $X$ beyond what is normally used for tree reconstruction~\cite{bar06}.
\end{enumerate}

Approach (ii) has received the most attention so far, with some positive results (for example, for reconstructing the subclass of normal networks from their induced trees \cite{wil11}). In this paper, we focus more on approach (iii), and, although we restrict to a class of subnetworks (which we call `orchard networks'), our reconstruction result has the additional strength that it can distinguish between any two networks from information on $X$ provided at least one of them is an orchard network. \blue{To provide some intuition, informally, a phylogenetic network is an orchard network if it can be reduced to a single vertex by recursively finding a pair of leaves that form either a cherry or a reticulated cherry, and then applying a cherry reduction to that pair of leaves.}

The type of information on $X$ we consider is the following. View the interior (non-leaf) vertices of a phylogenetic network $\cN$ as being labelled. In the biological setting, this label could correspond, for example, to the  genome of the ancestral species at this vertex (or some sub-genome that is sufficiently detailed to distinguish this ancestral vertex from others). For each species $x$ in the leaf set $X$, suppose we can \blue{count} the number of directed paths in the network from each ancestral genome (i.e.\ interior vertex) to $x$. This `ancestral profile'  is thus an ordered tuple of numbers, one tuple for each leaf in $X$ (note that current technology does not yet provide this information, so our approach is in the spirit of earlier mathematical results in phylogenetics that preceded the data required for their application). It turns out that such information is not enough to distinguish between an arbitrary pair of networks (we provide an example). However, if the underlying network $\cN$ is an orchard network, our main result shows that no other network (orchard or not) can have the same ancestral profile. Moreover, we present and justify a polynomial-time algorithm for reconstructing any orchard network from its ancestral profile. Our arguments rely on a structural property of orchard networks which also implies that there is a polynomial-time algorithm for testing whether or not an arbitrary network is an orchard network.

Our results generalise earlier work in~\blue{\cite{car08, car09}} which considered the more restricted \blue{classes} of \blue{`tree-sibling time-consistent' networks and} `tree-child' networks, \blue{respectively}. These authors use equivalent information on $X$ for reconstruction, however, their reconstruction result faces two limitations that are lifted here. First, the uniqueness \blue{results} of~\blue{\cite{car08, car09}} hold only within the class of \blue{tree-sibling time-consistent networks and} tree-child networks, whereas we show that ancestral profiles can distinguish an orchard network from any other network. Second, \blue{neither tree-sibling time-consistent networks nor} tree-child networks can have too many reticulate vertices (at most \blue{and $2n-4$ and} $n-1$, \blue{respectively}, where $n=|X|$), whereas orchard networks can have arbitrarily many reticulate vertices (independent of $n$).

Our results are also related to (and partly motivated by) earlier work by \cite{bar06} and \blue{\cite{wil08}} on `accumulation phylogenies'. This involved  a different subclass of networks (called `regular' in these papers, and `cluster networks' in \cite{hus10}), which neither contains, nor is contained in the subclass of orchard networks. A limitation of this subclass is that (unlike orchard networks)  they do not allow `redundant arcs' (an arc $(u, v)$ for which there is another path in the network from $u$ to $v$). Allowing redundant arcs has a strong biological motivation since even if each reticulation events happens instantaneously between two contemporaneous species, redundant arcs can still appear in the resulting network if not all species at the present are sampled. The results in~\cite{bar06, wil08} also assume any two networks being considered are within this same subclass.  In summary, our results are not directly related to this earlier work on accumulation phylogenies, apart from using a related type of information.

The paper is organised as follows. The next section contains some necessary definitions along with the statement of the main result (Theorem~\ref{main1}) and deduces, as a consequence, the main result (Theorem~1) in~\cite{car09}. This section also provides examples to justify various claims. Section~\ref{prelimsec} describes some preliminary lemmas, which apply more generally than for ancestral profiles, and in Section~\ref{ordersec} we state and prove the structural property of orchard networks that allows for an easy test as to whether or not an arbitrary network is of orchard type. The proof of Theorem~\ref{main1} is established in Section~\ref{proofsec}. \blue{We end the paper with a brief discussion in Section~\ref{last}.}

Lastly, just as we completed the write-up of this paper, a manuscript~\cite{jan18} was posted on arXiv that also considers the class of orchard networks (referred to as ``cherry-picking networks'' in~\cite{jan18}). The focus of that manuscript is quite different to that of this paper; nevertheless, it contains an independent and different proof of the structural property of orchard networks which is needed as a lemma for Theorem~\ref{main1} in this paper.

\section{Main Result}

Throughout the paper $X$ denotes a non-empty finite set and, unless otherwise stated, all paths are directed. For vertices $u$ and $v$ of a directed graph $D$, we say $v$ is {\em reachable} from $u$ if there is a path in $D$ from $u$ to $v$. Furthermore, for sets $A$ and $B$, we denote the set obtained from $A$ by removing every element in $A$ that is also in $B$ by $A-B$. If $|B|=1$, say $B=\{b\}$, we denote this by $A-b$.

\noindent {\bf Phylogenetic networks.} A {\em phylogenetic network on $X$} is a rooted acyclic directed graph with no arcs in parallel and satisfying the following properties:
\begin{enumerate}[(i)]
\item the (unique) root has \blue{in-degree zero and} out-degree two;

\item a vertex with out-degree zero has in-degree one, and the set of vertices with out-degree zero is $X$; and

\item all other vertices either have in-degree one and out-degree two, or in-degree two and out-degree one.
\end{enumerate}
For technical reasons, if $|X|=1$, we additionally allow a single vertex to be a phylogenetic network, in which case, the root is the vertex in $X$. Phylogenetic networks as defined here are also referred to as `binary phylogenetic networks' in the literature.

Let $\cN$ be a phylogenetic network on $X$. The vertices with out-degree zero are the {\em leaves} of $\cN$, and so $X$ is called the {\em leaf set} of $\cN$. Furthermore, vertices with in-degree one and out-degree two are {\em tree vertices}, while vertices of in-degree two and out-degree one are {\em reticulations}. The arcs directed into a reticulation are called {\em reticulation arcs}, all other arcs are {\em tree arcs}. To illustrate, an example of a phylogenetic network with leaf set $\{x_1, x_2, \ldots, x_6\}$ and three reticulations is shown in Fig.~\ref{orchard2}.

\begin{figure}
\center
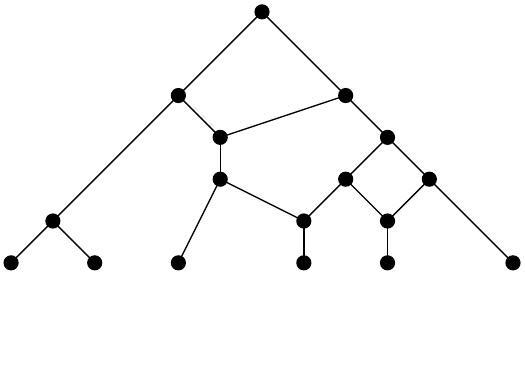
\caption{A phylogenetic network $\cN$ on $\{x_1, x_2, \ldots, x_6\}$. Here, $\{x_1, x_2\}$ is a cherry and $\{x_3, x_4\}$ is a reticulated cherry with $x_4$ the reticulation leaf.}
\label{orchard2}
\end{figure}

Lastly, let $\cN_1$ and $\cN_2$ be two phylogenetic networks on $X$ with vertex and arc sets $V_1$ and $E_1$, and $V_2$ and $E_2$, respectively. We say $\cN_1$ is {\em isomorphic} to $\cN_2$ if there exists a bijection $\varphi: V_1\rightarrow V_2$ such that $\varphi(x)=x$ for all $x\in X$, and $(u, v)\in E_1$ if and only if $(\varphi(u), \varphi(v))\in E_2$ for all $u, v\in V_1$.

\noindent {\bf Ancestral tuples and ancestral profile.} Let $\cN$ be a phylogenetic network on $X$ with vertex set $V$. Let $v_1, v_2, \ldots, v_t$ be a fixed (arbitrary) labelling of the vertices in $V-X$. For all $x\in X$, the {\em ancestral tuple} of $x$, denoted $\sigma(x)$, is the $t$-tuple whose $i$-th entry is the number of paths in $\cN$ from $v_i$ to $x$. Denoted by $\Sigma_{\cN}$, we call \red{the set}
$$\Sigma_{\cN}=\{\red{(x, \sigma(x))}: x\in X\},$$
\red{of ordered pairs} the {\em ancestral profile} of $\cN$. \blue{Furthermore, if $\cN'$ is a phylogenetic network on $X$ and, up to an ordering of the non-leaf vertices of $\cN'$, we have $\Sigma_{\cN'}=\Sigma_{\cN}$, we say {\em $\cN'$ realises $\Sigma_{\cN}$}. Lastly, although $\Sigma_{\cN}$ depends on the ordering of the vertices in $V-X$, the ordering is fixed and so the labelling can be effectively ignored.}

\noindent {\bf Cherries and reticulated cherries.} Let $\mathcal N$ be a phylogenetic network on $X$, and let $\{a, b\}$ be a $2$-element subset of $X$. Let $p_a$ and $p_b$ denote the parents of $a$ and $b$, respectively. We say $\{a, b\}$ is a {\em cherry} of $\mathcal N$ if $p_a=p_b$. Furthermore, if one of the parents, say $p_b$, is a reticulation and $(p_a, p_b)$ is an arc in $\mathcal N$, then $\{a, b\}$ is a {\em reticulated cherry} of $\cN$, in which case, $b$ is the {\em reticulation leaf} of the reticulated cherry. Observe that $p_a$ is necessarily a tree vertex. For the phylogenetic network shown in Fig.~\ref{orchard2}, $\{x_1, x_2\}$ is a cherry, while $\{x_3, x_4\}$ is a reticulated cherry in which $x_4$ is the reticulation leaf. \blue{Furthermore, in Fig.~\ref{orchard2}, $\{x_4, x_5\}$ is neither a cherry nor a reticulated cherry.}

We next describe two operations associated with cherries and reticulated cherries that are central to this paper. Let $\cN$ be a phylogenetic network. First suppose that $\{a, b\}$ is a cherry of $\cN$. Then {\em reducing} $b$ is the operation of deleting $b$ and suppressing the resulting vertex of in-degree one and out-degree one. If the parent of $a$ and of $b$ is the root of $\cN$, then reducing $b$ is the operation of deleting $b$ as well as deleting the root of $\cN$, thus leaving only the isolated vertex $a$. Now suppose that $\{a, b\}$ is a reticulated cherry of $\cN$ in which $b$ is the reticulation leaf. Then {\em cutting} $\{a, b\}$ is the operation of deleting the reticulation arc joining the parents of $a$ and $b$, and suppressing the two resulting vertices of in-degree one and out-degree one. It is easily seen that the operations of reducing a cherry and cutting a reticulated cherry both result in a phylogenetic network. Collectively, we refer to these two operations as {\em cherry reductions}. To illustrate, the phylogenetic network shown in Fig.~\ref{cutting}(i) (resp.\ Fig.~\ref{cutting}(ii)) has been obtained from the phylogenetic network in Fig.~\ref{orchard2} by reducing $x_2$ (resp.\ cutting $\{x_3, x_4\}$).

\begin{figure}
\center
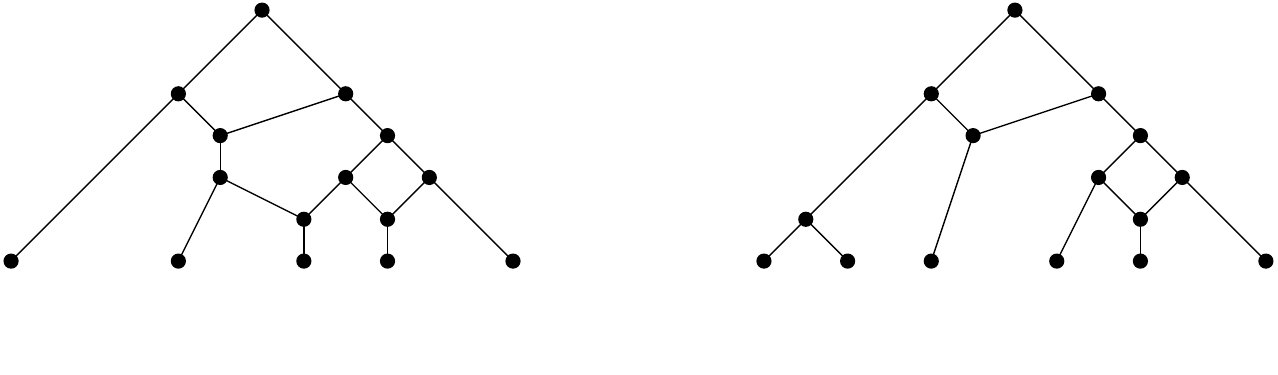
\caption{$\cN_1$ has been obtained from $\cN$ in Fig.~\ref{orchard2} by reducing $x_2$, while $\cN_2$ has been obtained from $\cN$ by cutting $\{x_3, x_4\}$.}
\label{cutting}
\end{figure}

\noindent {\bf Orchard networks.} For a phylogenetic network $\cN$, the sequence
\begin{align}
\cN=\cN_0, \cN_1, \cN_2, \ldots, \cN_k
\label{seq1}
\end{align}
of phylogenetic networks is a {\em cherry-reduction sequence of $\cN$} if, for all $i\in \{1, 2, \ldots, k\}$, the phylogenetic network $\cN_i$ is obtained from $\cN_{i-1}$ by a (single) cherry reduction. \blue{The sequence is {\em maximal} if $\cN_k$ has no cherries or reticulated cherries.} If $\cN_k$ consists of a single vertex, the sequence is {\em complete}, in which case, $\cN$ is called an {\em orchard network}. Observe that if (\ref{seq1}) is complete, then the leaf set of \blue{$\cN_{k-1}$} has size two and the parent of each leaf is the root of $\cN_{k-1}$. It is easily checked that the phylogenetic network shown in Fig.~\ref{orchard2} is an orchard network. In Section~\ref{ordersec}, we show that if $\cN$ is an orchard network, then every maximal sequence of cherry reductions of an orchard network $\cN$ is complete. Thus if we want to construct a complete cherry-reduction sequence for \blue{an orchard network}, the order in which the reductions are applied does not matter. In turn, this provides an easy test to decide whether or not an arbitrary network is orchard.

One of the most well-studied classes of phylogenetic networks is the class of tree-child networks. Introduced in~\cite{car09}, a phylogenetic network is {\em tree-child} if every non-leaf vertex is the parent of a tree vertex or a leaf. Tree-child networks are examples of orchard networks~\cite{bor16}, but there \blue{exist} orchard networks that are not tree-child. Indeed, while the size of the leaf set bounds the total number of vertices of a tree-child network~\cite{car09}, the total number of vertices in an orchard network is not necessarily bounded by the size of its leaf set. For example, the phylogenetic network shown in Fig.~\ref{orchard1}(i) is an orchard network with exactly three leaves but, by extending it in the obvious way, \blue{we} can \blue{produce an orchard network with} an arbitrarily large odd number of vertices \blue{and still with exactly three leaves}. Furthermore, not all phylogenetic networks are orchard networks as Fig.~\ref{orchard1}(ii) illustrates.

\begin{figure}
\center
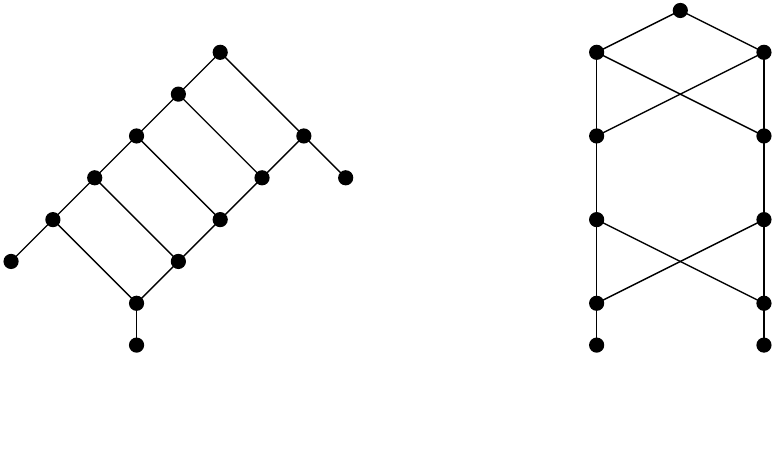
\caption{(i) An orchard network and (ii) a non-orchard network.}
\label{orchard1}
\end{figure}

\blue{For this paper, a second relevant class of phylogenetic networks is the class of tree-sibling time-consistent networks. Let $\cN$ be a phylogenetic network. We say $\cN$ is {\em tree-sibling} if every reticulation has a parent that is also the parent of a tree vertex or a leaf. Furthermore, $\cN$ is {\em time-consistent} if there is a map $t$ from the vertex set of $\cN$ to the non-negative integers such that if $(u, v)$ is a reticulation arc of $\cN$, then $t(u)=t(v)$; otherwise, $t(u)< t(v)$. We refer to such a mapping as a {\em temporal labelling}. In the literature, time-consistent networks are also referred to as {\em temporal} networks. Like tree-child networks, the class of tree-sibling time-consistent networks is a proper subclass of orchard networks. For completeness, we include a proof of containment. To see that it is proper, it is shown in~\cite{car08} that, unlike orchard networks, the number of reticulations of a tree-sibling time-consistent network is bounded by the size of its leaf set.}

\blue{
\begin{lemma}
Let $\cN$ be a tree-sibling time-consistent network. Then $\cN$ is an orchard network.
\end{lemma}}

\begin{proof}
\blue{Clearly, the lemma holds if $\cN$ has no reticulations. Therefore we may assume that $\cN$ has at least one reticulation. We first show that $\cN$ has either a cherry or a reticulated cherry. Let $t$ be a temporal labelling of the vertices of $\cN$, and let $v$ be a reticulation with the property that $t(v)\ge t(v')$ for all reticulations $v'$ of $\cN$. Since $\cN$ is tree-sibling, $v$ has a parent, $u$ say, that is the parent of a vertex $w$ which is either a tree vertex or a leaf. By maximality, no reticulations are reachable from $v$ or $w$. Therefore, if two leaves are reachable from either $v$ or $w$, then $\cN$ has a cherry. If this does not occur, then $w$ is a leaf and that the (unique) child, $x$ say, of $v$ is also a leaf. In particular, $\{w, x\}$ is a reticulated cherry of $\cN$.}

\blue{To complete the proof, let $\cN'$ be obtained from $\cN$ by a cherry reduction. Clearly, $\cN'$ is also tree-sibling. Furthermore, it is easily checked that the mapping $t'$ from the vertex set of $\cN'$ to the non-negative integers given by $t'(u)=t(u)$ is a temporal labelling of $\cN'$. Thus $\cN'$ is tree-sibling time-consistent. The lemma now follows.}
\end{proof}

\noindent \blue{{\bf Main result.}} The following theorem is the main result of the paper.

\begin{theorem}
Let $\cN$ be an orchard network on $X$ with vertex set $V$. Then, up to isomorphism, $\cN$ is the unique phylogenetic network on $X$ realising $\Sigma_{\cN}$. Furthermore, up to isomorphism, $\cN$ can be reconstructed from $\Sigma_{\cN}$ in time $O(|X|^3 |V|^3)$.
\label{main1}
\end{theorem}

It is worth emphasising that the uniqueness of $\cN$ in the statement of Theorem~\ref{main1} is amongst all phylogenetic networks on $X$, \blue{not just within the class of orchard networks on $X$}. Furthermore, if $\cN$ is not an orchard network, then the outcome of Theorem~\ref{main1} does not necessarily hold. In particular, consider the two phylogenetic networks $\cN_1$ and $\cN_2$ in Fig.~\ref{nonorchard}. It is easily checked that by fixing an ordering of the non-leaf vertices of each of $\cN_1$ and $\cN_2$ so that the parent of $y$ is in the same position in both orderings, we have $\Sigma_{\cN_1}=\Sigma_{\cN_2}$. But $\cN_1$ is not isomorphic to $\cN_2$.

\begin{figure}
\center
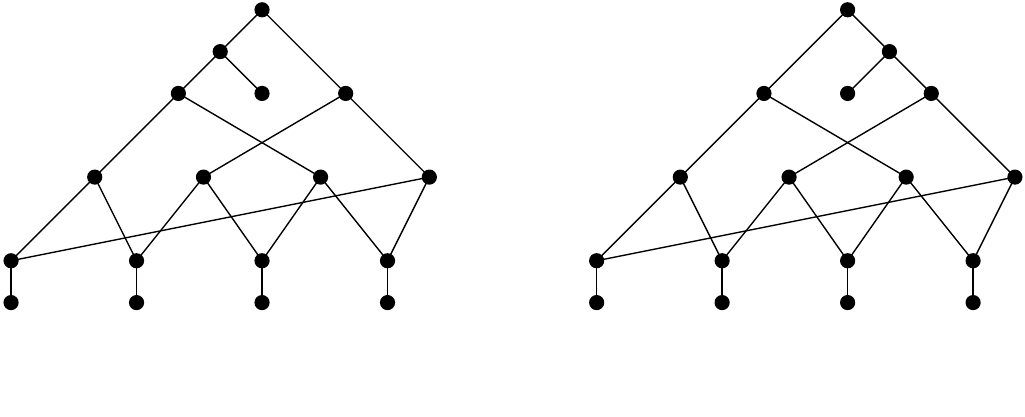
\caption{Two non-isomorphic phylogenetic networks $\cN_1$ and $\cN_2$, but $\Sigma_{\cN_1}=\Sigma_{\cN_2}$.}
\label{nonorchard}
\end{figure}

Theorem~\ref{main1} generalises \blue{results} of \blue{Cardona et al.~\cite{car08} and} Cardona et al.~\cite{car09}. Let $\cN$ be a phylogenetic network on $X$ with vertex set $V$ and let $x_1, x_2, \ldots, x_n$ be a fixed ordering of the leaves in $X$. For all $v\in V-X$, the {\em path tuple} of $v$, denoted $\pi(v)$, is the $n$-tuple whose $i$-th entry is the number of paths in $\cN$ from $v$ to $x_i$. Let $\Pi_{\cN}$ denote the multiset
$$\{\pi(v): v\in V-X\}$$
of path tuples of $\cN$. \blue{If $\cN'$ is a phylogenetic network on $X$ and, up to an ordering of $X$, we have $\Pi_{\cN'}=\Pi_{\cN}$, we say {\em $\cN'$ realises $\Pi_{\cN}$}.} The next theorem \blue{was} established in~\blue{\cite{car08} and}~\cite{car09}.

\begin{theorem}
Let $\cN$ be a phylogenetic network on $X$.
\begin{enumerate}[{\rm (i)}]
\item \blue{If $\cN$ is tree-sibling time-consistent, then, up to isomorphism, $\cN$ is the unique tree-sibling time-consistent network on $X$ realising $\Pi_{\cN}$.}

\item \blue{If $\cN$ is tree-child, then,} up to isomorphism, $\cN$ is the unique tree-child network on $X$ realising $\Pi_{\cN}$.
\end{enumerate}
Furthermore, \blue{for both instances,} up to isomorphism, $\cN$ can be constructed from $\Pi_{\cN}$ in time polynomial in the size of $X$.
\label{pathtuple}
\end{theorem}

Let $\cN$ be a phylogenetic network on $X$ with vertex set $V$. The \red{set} $\Sigma_{\cN}$ and \red{multiset} $\Pi_{\cN}$ are equivalent in the amount of information they provide. To see this, let $x_1, x_2, \ldots, x_n$ and $v_1, v_2, \ldots, v_t$ be fixed orderings of the vertices in $X$ and $V-X$, respectively. Then, for all $i\in \{1, 2, \ldots, t\}$, the $n$-tuple $\pi(v_i)$ is the tuple whose $j$-th entry is the $i$-th entry of $\sigma(x_j)$ for all $j\in \{1, 2, \ldots, n\}$. Similarly, each \red{ordered pair} in $\Sigma_{\cN}$ can be obtained from $\Pi_{\cN}$. Thus Theorem~\ref{main1} generalises Theorem~\ref{pathtuple} in two ways. First, it shows that the latter holds for the more general class of orchard networks and, second, the uniqueness is not confined to the class of networks \blue{being constructed}.

We end the section with \blue{three} remarks. Firstly, Theorem~\ref{main1} is not the first reconstruction result concerning the class of orchard networks. Although this class was not named, it is shown in~\cite{bor16} that orchard networks are reconstructible from their so-called multiset distance matrices. See~\cite[Theorem~3.4]{bor16}. We have no doubt that, over time, the class of orchard networks will be realised to be reconstructible in other ways as well.

The second remark concerns a related, but weaker, notion to that of ancestral tuples called ancestral sets. Let $\cN$ be a phylogenetic network on $X$ with vertex set $V$. For all $x\in X$, the {\em ancestral set} of $x$ is
$$\gamma(x)=\{v\in V-X:~\mbox{$x$ is reachable from $v$}\}.$$
Thus $\gamma(x)$ is the set of non-leaf vertices $v$ in $\cN$ for which there is a directed path from $v$ to $x$. Observe that, for all $x\in X$, the root of $\cN$ is always an element of $\gamma(x)$ and so $\gamma(x)$ is non-empty. Let $\Gamma_{\cN}$ denote the \red{set}
$$\{\red{(x, \gamma(x))}: x\in X\}$$
of \red{ordered pairs}. Given $\Sigma_{\cN}$, it is clear that we can construct $\Gamma_{\cN}$ in time $O(|V|)$.

\begin{figure}
\center
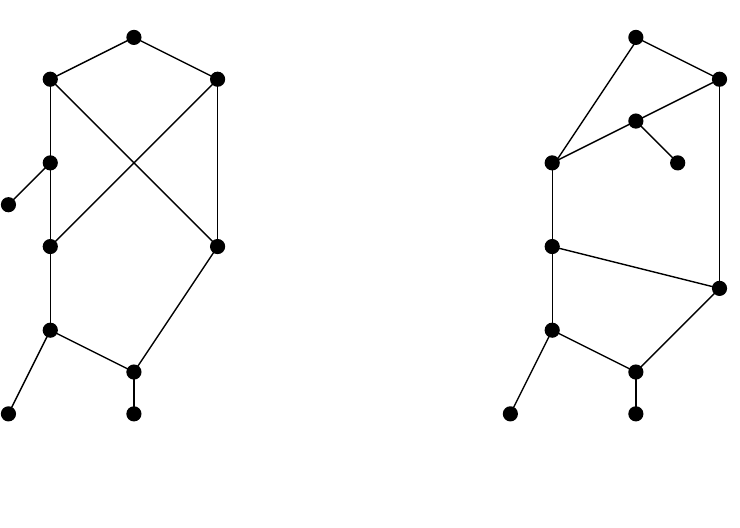
\caption{Two orchard networks $\cN_1$ and $\cN_2$ with $\Gamma_{\cN_1}=\Gamma_{\cN_2}$, but $\Sigma_{\cN_1}\neq \Sigma_{\cN_2}$.}
\label{sets}
\end{figure}

To see that ancestral sets is a weaker notion than ancestral tuples, consider the two orchard networks $\cN_1$ and $\cN_2$ shown in Fig.~\ref{sets}, where the non-leaf vertices have been labelled $1, 2, \ldots, 8$. For each $i\in \{1, 2\}$, the ancestral sets of $x_1$, $x_2$, and $x_3$ are $\{1, 2, 3, 4, 5, 7\}$, $\{1, 2, \ldots, 8\}$, and $\{1, 2, 3\}$, respectively. But $\cN_1$ is not isomorphic to $\cN_2$. Note that, for a fixed ordering of $1, 2, \ldots, 8$, the ancestral tuple of $x_2$ differs in $\cN_1$ and $\cN_2$ even though the ancestral tuples of $x_1$ and $x_3$ are the same for $\cN_1$ and $\cN_2$. Nevertheless, despite this example, the ancestral sets of a phylogenetic network $\cN$ do provide some information regarding the structure of $\cN$. As this is of possible independent interest, we highlight this in the next section where the preliminary lemmas are established in terms of ancestral sets.

\blue{The third remark concerns the relationship between orchard networks and the increasingly prominent class of tree-based networks~\cite{fra15}. A phylogenetic network $\cN$ on $X$ with root $\rho$ and vertex set $V$ is {\em tree-based} if it has, as a subgraph, a rooted subtree with root $\rho$, vertex set $V$, and leaf set $X$. Note that $\rho$ in the subtree may have out-degree one. It is shown in~\cite{hub19} that the class of orchard networks is a proper subclass of tree-based networks. To see that it is proper, observe that the non-orchard networks $\cN_1$ and $\cN_2$ in Fig.~\ref{nonorchard} are both tree-based. Thus, the networks in this figure also show that Theorem~\ref{main1} does not extend to tree-based networks.}

\section{Preliminary Lemmas}
\label{prelimsec}

In this section, we establish several results that will be used in the proof of Theorem~\ref{main1}. These results show \blue{that} the ancestral sets, and thus the ancestral tuples, of an arbitrary phylogenetic network recognise and distinguish cherries and reticulated cherries.

\begin{lemma}
Let $\cN$ be a phylogenetic network on $X$, and let $a$ and $b$ be distinct elements in $X$. Then $\gamma(a)\subseteq \gamma(b)$ if and only if the parent of $b$ is reachable from the parent of $a$.
\label{parents}
\end{lemma}

\begin{proof}
Let $p_a$ and $p_b$ denote the parents of $a$ and $b$, respectively. If $p_b$ is reachable from $p_a$, then it is clear that $\gamma(a)\subseteq \gamma(b)$. To prove the converse, suppose that $\gamma(a)\subseteq \gamma(b)$. Then $p_a\in \gamma(b)$ and so, by definition, $b$ is reachable from $p_a$. In turn, this implies that $p_b$ is reachable from $p_a$.
\end{proof}

The next corollary immediately follows from Lemma~\ref{parents} and the fact that phylogenetic networks are acyclic.

\begin{corollary}
Let $\cN$ be a phylogenetic network on $X$, and let $\{a, b\}$ be a $2$-element subset of $X$. Then $\{a, b\}$ is a cherry in $\cN$ if and only if $\gamma(a)=\gamma(b)$.
\label{cherry1}
\end{corollary}

%
%

\begin{lemma}
Let $\cN$ be a phylogenetic network on $X$, and let $\{a, b\}$ be a $2$-element subset of $X$. Then $\{a, b\}$ is a reticulated cherry of $\cN$ in which $b$ is the reticulation leaf if and only if
\begin{enumerate}[{\rm (i)}]
\item $\gamma(a)\blue{\subsetneq} \gamma(b)$,

\item there is no $x\in X-b$ such that $\gamma(a)\subset \gamma(x)$, and

\item $\left|\gamma(b)-\bigcup_{x\in X-b} \gamma(x)\right|=1$.
\end{enumerate}
\label{cherry3}
\end{lemma}

\begin{proof}
Let $p_a$ and $p_b$ denote the parents of $a$ and $b$, respectively. It is easily checked that if $\{a, b\}$ is a reticulated cherry in which $b$ is the reticulation leaf, then (i)--(iii) hold. So suppose that (i)--(iii) hold. Since (i) holds, it follows by Lemma~\ref{parents} that there is a directed path $P$ in $\cN$ from $p_a$ to $p_b$. If $p_b$ is a tree vertex, then $\cN$ has a leaf, \blue{$c$} say, reachable from $p_b$ such that $\blue{c}\neq b$. This implies that $\gamma(a)\subset \gamma(\blue{c})$, contradicting (ii). Therefore $p_b$ is a reticulation. Lastly, assume $(p_a, p_b)$ is not an arc in $\cN$. Let $u$ denote the vertex on $P$ immediately prior to $p_b$. If $u$ is a tree vertex, then $\cN$ has a leaf $\blue{c'}\neq b$ reachable from $u$ with $\gamma(a)\subset \gamma(\blue{c'})$, contradicting (ii). On the other hand, if $u$ is a reticulation, then
$$\left|\gamma(b)-\bigcup_{x\in X-b} \gamma(x)\right|\ge 2,$$
contradicting (iii). Thus $(p_a, p_b)$ is an arc and so $\{a, b\}$ is a reticulated cherry in which $b$ is the reticulation leaf.
\end{proof}

\section{Order Does Not Matter}
\label{ordersec}

Let $\cN$ be an orchard network. Then, by definition, there exists a complete cherry-reduction sequence for $\cN$. But, how do we find such a sequence and does the order in which we apply the cherry reductions matter? The next proposition says that if we take $\cN$ and repeatedly apply cherry reductions until no more is possible, we always construct a complete cherry-reduction sequence. \blue{A vertex on a directed path is {\em non-terminal} if it is neither the first nor last vertex on the path.}

\begin{proposition}
Let $\cN$ be an orchard network, and let
\begin{align}
\cN=\cN_0, \cN_1, \cN_2, \ldots, \cN_{\ell}
\label{seq2}
\end{align}
be a maximal sequence of cherry reductions. Then this sequence is complete.
\label{order1}
\end{proposition}

\begin{proof}
Let $X$ denote the leaf set of $\cN$, and suppose (\ref{seq2}) is not complete. Paralleling (\ref{seq2}), we begin by constructing a sequence
$$\cN=\cM_0, \cM_1, \cM_2, \ldots, \cM_{\ell}$$
of rooted acyclic directed graphs as follows. If $\cN_1$ is obtained from $\cN_0$ by reducing a leaf of a cherry, then $\cM_1$ is obtained from $\cM_0$ by deleting the same leaf but not suppressing the resulting vertex of in-degree one and out-degree one. Similarly, if $\cN_1$ is obtained from $\cN_0$ by cutting a reticulated cherry, then $\cM_1$ is obtained from $\cM_0$ by deleting the same reticulation arc but not suppressing the two resulting vertices of in-degree one and out-degree one. More generally, if $\cN_i$ is obtained from $\cN_{i-1}$ by reducing a leaf of a cherry, that is, deleting a leaf $b$ say and suppressing its parent $p_b$, then $\cM_i$ is obtained from $\cM_{i-1}$ by deleting $b$ as well as deleting every non-terminal vertex on the (unique) path from $p_b$ to $b$ in $\cM_{i-1}$. Note that each of these non-terminal vertices has in-degree one and out-degree one in $\cM_{i-1}$. On the other hand, if $\cN_i$ is obtained from $\cN_{i-1}$ by cutting a reticulated cherry, that is, deleting a reticulation arc $(p_a, p_b)$ and suppressing $p_a$ and $p_b$, then $\cM_i$ is obtained from $\cM_{i-1}$ by deleting $(p_a, p_b)$. Observe that, for all $i$, if we suppress every vertex in $\cM_i$ of in-degree one and out-degree one, we obtain $\cN_i$. Thus $\cM_i$ is a subdivision of $\cN_i$ for all $i$, \blue{that is, $\cN_i$ can be obtained from $\cM_i$ by suppressing all vertices of in-degree one and out-degree one for all $i$}. Furthermore, as (\ref{seq2}) is not complete, the root $\rho$ of $\cN$ is never deleted and so, for all $i$, the root of $\cM_i$ is also $\rho$ and has out-degree two in $\cM_i$.

We now analyse $\cM_{\ell}$. Since (\ref{seq2}) is maximal and not complete, $\cN_{\ell}$ has at least one reticulation. This implies that $\cM_{\ell}$ has at least one vertex of in-degree two and out-degree one. We next show that every non-terminal vertex in $\cM_{\ell}$ on a path from $\rho$ to a vertex of in-degree two and out-degree one has degree three.

\begin{sublemma}
Let $v$ be a vertex of in-degree two and out-degree one in $\cM_{\ell}$. If $u$ is a non-terminal vertex of $\cM_{\ell}$ on a path in $\cM_{\ell}$ from $\rho$ to $v$, then $u$ has \blue{degree three} in $\cM_{\ell}$.
\label{three}
\end{sublemma}

\begin{proof}
Suppose $u$ is a vertex of in-degree one and out-degree one on a path from $\rho$ to $v$ in $\cM_{\ell}$. In $\cN$, the vertex $u$ has degree three. Therefore, for some $i\in \{1, 2, \ldots, \ell\}$, we have that $\cN_i$ is obtained from $\cN_{i-1}$ by a cherry reduction in which an arc incident with $u$ is deleted. Now, as $v$ is a vertex of in-degree two and out-degree one in $\cM_{\ell}$, it follows that $v$ is a reticulation in $\cN_{\ell}$, and therefore a reticulation in $\cN_i$. Thus there is a path $P$ in $\cN_i$ from $u$ to $v$. It is now easily checked that no cherry reduction applied to $\cN_{i-1}$ in which an arc incident with $u$ and not lying on $P$ is deleted is possible. Hence $u$ has degree-three.
\end{proof}

We now complete the proof of the proposition. Since $\cN$ is orchard, there is a sequence
\begin{align*}
\cN=\cN'_0, \cN'_1, \cN'_2, \ldots, \cN'_k
\end{align*}
of cherry reductions such that $\cN'_k$ consists of a single vertex. Let $i$ be the smallest index such that $\cN'_i$ is obtained from $\cN'_{i-1}$ by cutting a reticulated cherry in which the deleted reticulation arc, $(u, v)$ say, has the property that $v$ is in $\cM_{\ell}$ and it has in-degree two and out-degree one in $\cM_{\ell}$. Observe that, by the choice of $i$, no vertex of in-degree two and out-degree one is reachable from $v$ in $\cM_{\ell}$ except $v$ itself. As (\ref{seq2}) is maximal, this implies that there is a unique vertex, $\ell_v$ say, in $X$ that is reachable from $v$ in $\cM_{\ell}$.

Now, $u$ is a tree vertex in $\cN'_{i-1}$ whose other child, in addition to $v$, is a leaf. By~(\ref{three}), $u$ has degree-three in $\cM_{\ell}$. Furthermore, as $u$ is a tree vertex in $\cN'_{i-1}$, it follows that $u$ has in-degree one and out-degree two in $\cM_{\ell}$. Let $w$ denote the child of $u$ in $\cM_{\ell}$ that is not $v$. At least one vertex in $X$ is reachable from $w$ in $\cM_{\ell}$ and this vertex is not $\ell_v$. If, in $\cM_{\ell}$, there is no vertex reachable from $w$ with in-degree two and out-degree one, then (\ref{seq2}) is not maximal. Therefore, in $\cM_{\ell}$ there is such a vertex $w'$ reachable from $w$. In $\cN$, the vertex $w'$ is a reticulation, and so there is a $j\in \{1, 2, \ldots, k\}$ such that $\cN'_j$ is obtained from $\cN'_{j-1}$ by cutting a reticulated cherry in which a reticulation arc directed into $w'$ is deleted. Since $(u, v)$ is the reticulation arc directed into $v$ that is deleted, it follows $j< i$. But, by the choice of $i$, we have $i< j$; a contradiction. We conclude that (\ref{seq2}) is complete.
\end{proof}

The following corollary is an immediate consequence of Proposition~\ref{order1}.

\begin{corollary}
Let $\cN$ be an orchard network, and let $\{a, b\}$ be a cherry or a reticulated cherry of $\cN$. If $\cN'$ is obtained from $\cN$ by reducing $b$ if $\{a, b\}$ is a cherry or cutting $\{a, b\}$ if $\{a, b\}$ is a reticulated cherry, then $\cN'$ is an orchard network.
\label{order2}
\end{corollary}

Since deciding if a given pair of leaves of a phylogenetic network is either a cherry or a reticulated cherry takes constant time and a cherry reduction also takes constant time, the last corollary gives a polynomial-time algorithm for deciding if an arbitrary phylogenetic network $\cN$ is orchard. In particular, repeatedly find a cherry or a reticulated cherry, and apply the appropriate cherry reduction until this process is no longer possible. This takes at most $O(|V|)$ iterations, where $V$ is the vertex of $\cN$. If at the completion of this process, we have a phylogenetic network consisting of a single vertex, then $\cN$ is orchard; otherwise, $\cN$ is not orchard. \blue{Observe that if $\cN$ is orchard with $n$ leaves and $k$ reticulations, then this process consists of $n+k-1$ cherry reductions.}

\section{Proof of Theorem~\ref{main1}}
\label{proofsec}

In this section, we prove Theorem~\ref{main1}. For a phylogenetic network $\cN$, Corollary~\ref{cherry1} and Lemma~\ref{cherry3} show that it is straightforward to recognise cherries and reticulated cherries of $\cN$ using only the ancestral sets, and thus the ancestral tuples, of $\cN$. This fact is freely used throughout this section. We next describe two operations on tuples that parallel the operations of reducing a cherry and cutting a reticulated cherry.

Let $X$ be a non-empty finite set and, for some fixed $t$, let
$$\Sigma=\{\red{(x, \sigma(x))}: x\in X\}$$
be a \red{set of ordered pairs, where, for all $x\in X$, we have that $\sigma(x)$ is a $t$-tuple} whose entries are either non-negative integers or $-$. Note that the symbol $-$ is going to be used as a placeholder. Let $\{a, b\}$ be a $2$-element subset of $X$. The first operation will be used only in association with reducing $b$ when $\{a, b\}$ is a cherry. Let $j\in \{1, 2, \ldots, t\}$ such that $\sigma_j(a)=\sigma_j(b)=1$, but $\sigma_j(x)=0$ for all $x\in X-\{a, b\}$. Let $\Sigma'$ be the \red{set of $|X-b|$ ordered pairs} obtained from $\Sigma$ \red{as follows.} For all $x\in X-b$, \red{set $\sigma'(x)$ so that the $i$-th entry is}
$$\sigma'_i(x)=
\begin{cases}
\sigma_i(x), & \mbox{if $i\neq j$;} \\
-, & \mbox{if $i=j$.}
\end{cases}
$$
\red{Set $\Sigma'=\{(x, \sigma'(x)): x\in X-b\}$.} We say that $\Sigma'$ has been obtained from $\Sigma$ by {\em reducing $b$}.

The second operation will be used only in association with cutting $\{a, b\}$ when $\{a, b\}$ is a reticulated cherry in which $b$ is the reticulation leaf. Let $j\in \{1, 2, \ldots, t\}$ such that $\sigma_j(a)=1=\sigma_j(b)$ but $\sigma_j(x)=0$ for all $x\in X-\{a, b\}$, and let $k\in \{1, 2, \ldots, t\}$ such that $\sigma_k(b)=1$ but $\sigma_k(x)=0$ for all $x\in X-b$. Let $\Sigma'$ be the \red{set of $|X|$ ordered pairs} obtained from $\Sigma$ \red{as follows.} For all $x\in X-b$, \red{set $\sigma'(x)$ so that the $i$-th entry is}
$$\sigma'_i(x)=
\begin{cases}
\sigma_i(x), & \mbox{if $i\not\in \{j, k\}$;} \\
-, & \mbox{if $i\in \{j, k\}$;}
\end{cases}
$$
and \red{set $\sigma'(b)$ so that the $i$-th entry is}
$$\sigma'_i(b)=
\begin{cases}
\sigma_i(b)-\sigma_i(a), & \mbox{if $i\not\in \{j, k\}$;} \\
-, & \mbox{if $i\in \{j, k\}$.}
\end{cases}
$$
\red{Set $\Sigma'=\{(x, \sigma'(x): x\in X\}$.} We say that $\Sigma'$ has been obtained from $\Sigma$ by {\em cutting $\{a, b\}$}.

\begin{lemma}
Let $\cN$ be a phylogenetic network on $X$ \blue{with vertex set $V$ and} $|X|\ge 2$, \blue{and fix an ordering of $V-X$}. Let $\{a, b\}$ be a $2$-element subset of $X$.
\begin{enumerate}[{\rm (i)}]
\item If $\{a, b\}$ is a cherry of $\cN$, then, up to \blue{entries with symbol $-$}, the \red{set of ordered pairs} obtained from $\Sigma_{\cN}$ by reducing $b$ is the ancestral profile of the phylogenetic network $\cN'$ obtained from $\cN$ by reducing $b$.

\item If $\{a, b\}$ is a reticulated cherry of $\cN$ in which $b$ is the reticulation leaf, then, up to \blue{entries with symbol $-$}, the \red{set of ordered pairs} obtained from $\Sigma_{\cN}$ by cutting $\{a, b\}$ is the ancestral profile of the phylogenetic network $\cN'$ obtained from $\cN$ by cutting $\{a, b\}$.
\end{enumerate}
\label{reduce}
\end{lemma}

\begin{proof}
We prove the lemma for (ii). The proof of the lemma for (i) is similar, but easier, and omitted. Suppose $\{a, b\}$ is a reticulated cherry of $\cN$ in which $b$ is the reticulation leaf, and $\cN'$ is obtained from $\cN$ by cutting $\{a, b\}$. Let $\Sigma'$ be the \red{set of ordered pairs} obtained from $\Sigma_{\cN}$ by cutting $\{a, b\}$. We will show that $\Sigma'$ is the ancestral profile of a phylogenetic network isomorphic to $\cN'$.

Let $V$ denote the vertex set of $\cN$, and fix an ordering $v_1, v_2, \ldots, v_t$ of the vertices in $V-X$. Let $p_a$ and $p_b$ denote the parents of $a$ and $b$, respectively, in $\cN$. Set
$$U_a=\{v_j\in V-X:~\mbox{$\sigma_j(a)=1=\sigma_j(b)$, $\sigma_j(x)=0$ for all $x\in X-\{a, b\}$}\}$$
and
$$U_b=\{v_k\in V-X:~\mbox{$\sigma_k(b)=1$, $\sigma_k(x)=0$ for all $x\in X-b$}\}.$$
Observe that $U_a$ and $U_b$ are both non-empty as $p_a\in U_a$ and $p_b\in U_b$, but $U_a\cap U_b$ is empty.

Now consider $\Sigma'$. To obtain $\Sigma'$ from $\Sigma_{\cN}$, we chose (i) an entry in $\sigma(a)$, say $j$, such that $\sigma_j(a)=1=\sigma_j(b)$ but $\sigma_j(x)=0$ for all $x\in X-\{a, b\}$, and (ii) an entry in $\sigma(b)$, say $k$, such that $\sigma_k(b)=1$ but $\sigma_k(x)=0$ for all $x\in X-b$. In particular, these chosen entries correspond to vertices, $v_j$ and $v_k$ say, in $U_a$ and $U_b$, respectively.

Let $\cN_1$ denote the phylogenetic network obtained from $\cN$ by bijectively relabelling the vertices in $U_a$ with the vertices in $U_a$ so that $p_a$ is relabelled $v_j$, and bijectively relabelling the vertices in $U_b$ with the vertices in $U_b$ so that $p_b$ is relabelled $v_k$. Clearly, $\cN_1$ is isomorphic to $\cN$ and $\Sigma_{\cN}$ is the ancestral profile of $\cN_1$. Furthermore, it is easily checked that, up to isomorphism, $\Sigma'$ is the ancestral profile of the phylogenetic network $\cN'_1$ obtained from $\cN_1$ by cutting $\{a, b\}$. But $\cN'_1$ is isomorphic to $\cN'$, thereby completing the proof of the lemma.
\end{proof}

With Lemma~\ref{reduce} in hand, we next prove the uniqueness part of Theorem~\ref{main1}

\begin{proof}[Proof of the uniqueness part of Theorem~\ref{main1}.]
The proof is by induction on the sum of the number $n$ of leaves and the number \blue{$k$} of reticulations in $\cN$. If $n+k=1$, then $n=1$ and \blue{$k=0$}, and $\cN$ consists of the single vertex in $X$, and so uniqueness holds. If $n+k=2$, then, as $\cN$ is orchard, $n=2$ and $k=0$, in which case, $\cN$ consists of two leaves attached to the root. Again, uniqueness holds. Now suppose that $n+k\ge 3$ and the uniqueness holds for all orchard networks for which the sum of the number of leaves and the number of reticulations is at most $n+k-1$. Note that, as $\cN$ is orchard, $n\ge 2$.

Since $\cN$ is orchard, it has either a cherry or a reticulated cherry. Thus, by Corollary~\ref{cherry1} and Lemma~\ref{cherry3}, it is possible to find a $2$-element subset $\{a, b\}$ of $X$ using only $\Sigma_{\cN}$ such that $\{a, b\}$ is either a cherry or a reticulated cherry of $\cN$. If the latter, we can also determine from $\Sigma_{\cN}$ which of $a$ and $b$ is the reticulation. Without loss of generality, we may assume $b$ is the reticulation leaf. Depending on whether $\{a, b\}$ is a cherry or a reticulated cherry, let $\cN'$ be obtained from $\cN$ by reducing $b$ or cutting $\{a, b\}$, respectively, and let $\Sigma'$ be the \red{set of ordered pairs} obtained from $\Sigma_{\cN}$ by reducing $b$ or cutting $\{a, b\}$, respectively. Regardless of the way $\cN'$ and $\Sigma'$ are obtained, it follows by Corollary~\ref{order2} and Lemma~\ref{reduce} that $\cN'$ is an orchard network and, up to isomorphism, $\Sigma'$ is the ancestral profile of $\cN'$. Furthermore, $\cN'$ has either $n-1$ leaves and $k$ reticulations if $\{a, b\}$ is a cherry, or $n$ leaves and $k-1$ reticulations if $\{a, b\}$ is a reticulated cherry. Therefore, by the induction assumption, up to isomorphism, $\cN'$ is the unique phylogenetic network whose ancestral profile is $\Sigma'$.

Now let $\cN_1$ be a phylogenetic network on $X$ such that $\Sigma_{\cN}$ is the ancestral profile of $\cN_1$. Note that $\cN_1$ has the same number of non-leaf vertices as $\cN$, but not necessarily the same number of reticulations. First assume $\{a, b\}$ is a cherry of $\cN$. Then, by Corollary~\ref{cherry1}, $\{a, b\}$ is a cherry of $\cN_1$. Let $\cN'_1$ denote the phylogenetic network obtained from $\cN_1$ by reducing $b$. By Lemma~\ref{reduce}(i), up to isomorphism, $\Sigma'$ is the ancestral profile of $\cN'_1$. Thus, by the induction assumption, $\cN'_1$ is isomorphic to $\cN'$. Since $\{a, b\}$ is a cherry of $\cN_1$ and $\cN$, it follows that $\cN_1$ is isomorphic to $\cN$. 

Lastly, assume $\{a, b\}$ is a reticulated cherry of $\cN$. Then, by Lemma~\ref{cherry3}, $\{a, b\}$ is a reticulated cherry of $\cN_1$ in which $b$ is the reticulation leaf. Let $\cN'_1$ be the phylogenetic network obtained from $\cN_1$ by cutting $\{a, b\}$. By Lemma~\ref{reduce}(ii), up to isomorphism, $\Sigma'$ is the ancestral profile of $\cN'_1$. Hence, by the induction assumption, $\cN'_1$ is isomorphic to $\cN'$. As $\{a, b\}$ is a reticulated cherry of $\cN$ and $\cN_1$ in which $b$ is the reticulation leaf, we have that $\cN_1$ is isomorphic to $\cN$. This completes the proof of the uniqueness part of Theorem~\ref{main1}.
\end{proof}

\subsection{The algorithm}

Let $\cN$ be an orchard network on $X$, and let $\Sigma$ denote the ancestral profile of $\cN$. Called {\sc Orchard Tuple}, we next describe an algorithm which takes as its input $X$ and $\Sigma$, and returns a phylogenetic network $\cN_1$ on $X$ that is isomorphic to $\cN$. The proof that the algorithm works correctly is essentially the same as that used to prove the uniqueness part of Theorem~\ref{main1}, and so it is omitted. The running time of the algorithm follows its description.

\begin{enumerate}[1.]
\item If $|X|=1$, then return the phylogenetic network consisting of the single vertex in $X$.

\item Else, find a $2$-element subset, $\{a, b\}$ say, of $X$ such that either (I) $\gamma(a)=\gamma(b)$ or (II) $\gamma(a)\subset \gamma(b)$, there is no $x\in X-b$ with $\gamma(a)\subseteq \gamma(x)$, and
$$\left|\gamma(b)-\bigcup_{x\in X-b} \gamma(x)\right|=1.$$

\begin{enumerate}[(a)]
\item If $\{a, b\}$ satisfies (I) (in which case $\{a, b\}$ is a cherry), then
\begin{enumerate}[(i)]
\item Reduce $b$ in $\Sigma$ to give the \red{set} $\Sigma'$ of \red{$|X-b|$ ordered pairs}.

\item Apply {\sc Orchard Tuple} to input \red{$X'=X-b$} and $\Sigma'$. Construct $\cN_1$ from the returned phylogenetic network $\cN'_1$ on $X'$ by subdividing the arc incident to $a$ with a new vertex $p_a$, and adjoining a new leaf $b$ via the new arc $(p_a, b)$. \blue{If $|X'|=1$, then set $\cN_1$ to be the phylogenetic network consisting of the leaves $a$ and $b$ adjoined to the root.} Return $\cN_1$.
\end{enumerate}

\item Else, $\{a, b\}$ satisfies (II) (in which case $\{a, b\}$ is a reticulated cherry and $b$ is the reticulation leaf).
\begin{enumerate}[(i)]
\item Cut $\{a, b\}$ in $\Sigma$ to give the \red{set} $\Sigma'$ of \red{$|X|$ ordered pairs}.

\item Apply {\sc Orchard Tuple} to $X$ and $\Sigma'$. Construct $\cN_1$ from the returned phylogenetic network $\cN'_1$ on $X$ by subdividing the arcs incident to $a$ and $b$ with new vertices $p_a$ and $p_b$, respectively, and adding the new arc $(p_a, p_b)$. Return $\cN_1$.
\end{enumerate}
\end{enumerate}
\end{enumerate}

We now consider the running time of {\sc Orchard Tuple}. The input to the algorithm is a set $X$ and the ancestral profile of an orchard network $\cN$ on $X$ whose entries are either a non-negative integer or the symbol $-$. Let $V$ denote the vertex set of $\cN$. As noted earlier, the \red{set $\Gamma_{\cN}=\{(x, \gamma(x)): x\in X\}$} can be determined from $\Sigma$ in $O(|V|)$ time. This is a preprocessing step and it will have no effect on the theoretical running time. Except for when $|X|\in \{1, 2\}$, in which case, {\sc Orchard Tuple} runs in constant time, each iteration begins by finding a $2$-element subset of $X$ satisfying either (I) or (II). This takes $O(|X|^2 |V|)$ time as there are $O(|X|^2)$ two-element subsets of $X$ and each subset takes $O(|V|)$ time to decide if is satisfies either (I) or (II). Once such a $2$-element is found, we construct $\Sigma'$. Regardless of the way $\Sigma'$ is constructed, this takes $O(|X| |V|)$ time. When $\cN'_1$ is returned, we augment to $\cN_1$ in constant time, and so each iteration takes $O(|X|^3 |V|^2)$ time.

When we recurse, $\Sigma'$ is the ancestral profile of an orchard network with either one less leaf or one less reticulation than an orchard network for which $\Sigma$ is the ancestral profile. Thus the total number of iterations is $O(|V|)$. We conclude that {\sc Orchard Tuple} completes in $O(|X|^3 |V|^3)$ time. This completes the proof of Theorem~\ref{main1}.

\section{Conclusion}
\label{last}

\blue{The main result of this paper, Theorem~\ref{main1}, shows that the ancestral profile of an orchard network $\cN$ on $X$ uniquely determines $\cN$ amongst all phylogenetic networks on $X$. This generalises results in both~\cite{car08} and~\cite{car09}, which considered tree-sibling time-consistent networks and tree-child networks (subclasses of orchard networks whose number of reticulations is at most linear in the number of leaves). Curiously, these later results have a different motivation compared to what motivated Theorem~\ref{main1}. There the motivation is to construct a distance measure (metric) on the classes of tree-sibling time-consistent networks and tree-child networks which is computable in polynomial time. Recalling that they considered the equivalent notion of path-tuples, for two tree-sibling time-consistent (resp.\ tree-child) networks $\cN_1$ and $\cN_2$, the distance between $\cN_1$ and $\cN_2$ is the value}
$$\blue{\left|\Pi_{\cN_1}\triangle \Pi_{\cN_2}\right|,}$$
\blue{where the symmetric difference and the cardinality operator refer to multisets. It is easily checked that this same measure extends to the class of orchard networks.}

\blue{As noted in the introduction, our result does not relate to specific biological data that is readily available at present. However, a type of data that might provide ancestral profile information would be genomic fragments that follow lineage splitting and reticulation events, so that when a reticulation occurs, a trace of each fragment in the incoming lineage is preserved in (different regions of) the reticulate genome.}

\blue{Lastly, we end with a question asked by one of the referees. For a given orchard network $\cN$, is it possible to count the number of complete cherry-reduction sequences of $\cN$?}

\blue{\section*{Acknowledgements} We thank the three anonymous referees for their careful reading of the paper and constructive comments.}

\end{document}